\documentclass[11pt]{amsart}

\usepackage{microtype}

\usepackage[T1]{fontenc}

\usepackage{RNdefn}
\usepackage{eucal, mathtools}
\usepackage{tikz}
\usepackage{tikz-cd}
\usepackage{aliascnt}

\usepackage[nameinlink]{cleveref}
\crefname{subsection}{\S\!\!}{subsections}
\crefname{section}{\S\!\!}{\S\!}
\renewcommand{\autoref}{\Cref}

\newcommand{\St}{\mathbf{St}}

\newcommand{\VB}{\mathrm{VB}}

\newcommand{\m}{\to}

\newaliascnt{thmauto}{thmcounter}

\newaliascnt{Defauto}{thmcounter}

\newaliascnt{exauto}{thmcounter}

\newaliascnt{exsauto}{thmcounter}

\newaliascnt{lemauto}{thmcounter}

\newaliascnt{propauto}{thmcounter}

\newaliascnt{corauto}{thmcounter}

\newaliascnt{conjauto}{thmcounter}

\newaliascnt{queauto}{thmcounter}

\newaliascnt{remauto}{thmcounter}

\usepackage[T1]{fontenc}
\usepackage{lmodern}

\theoremstyle{plain}
\newtheorem{theorem}[thmauto]{Theorem}

\newtheorem{lemma}[lemauto]{Lemma}
\newtheorem{proposition}[propauto]{Proposition}

\newtheorem{remark}[remauto]{Remark}

\newtheorem*{rem*}{Remark}
\newtheorem*{thm*}{Theorem}
\newtheorem*{exs*}{Examples}

\newtheorem*{definition*}{Definition}

	\title[Corrigendum]{Corrigendum to ``Stability in the high-dimensional cohomology of congruence subgroups'' [Compos. Math. 156 (2020), no. 4, 822--861]}

\author{Jeremy Miller}
\address{150 N University Street
	West Lafayette, IN-47904
	USA, Department of Mathematics, Purdue University, USA}
\email{jeremykmiller@purdue.edu}
\author{Rohit Nagpal}
\author{Peter Patzt}
\address{601 Elm Av,
	Norman, OK-73019
	USA, Department of Mathematics, University of Oklahoma, USA}
\email{ppatzt@ou.edu}
\date{\today}

\begin{document}
	
	
\maketitle

After the publication of \cite{MNP}, Andrew Putman pointed out a mistake in our paper and helped us fix it. In this note, we will explain what this mistake is and how to fix it. We will assume familiarity with the paper here.

The problem is in \cite[Section 3.2]{MNP}, where we prove \cite[Theorem 3.2]{MNP}, that the Steinberg monoid is Koszul. In particular, the partial order defined on page 833 of \cite{MNP} does not satisfy (P2) as we claim without proof. We will correct the proof by giving a partial order that satisfies all claimed properties and proving them. For completeness sake, we give a complete proof of \cite[Proposition 3.6]{MNP}, which states the following.

\begin{proposition} \label{prop3.6}
\[ H_s(\overline{\mathcal B}^n_*( \St)) = 0 \quad\text{if $s\neq n$.}\]
\end{proposition}

We begin by recalling some essential definitions and properties from the paper. For a description of the elements of  Steinberg modules, we recall \cite[Theorem 2.4]{MNP} which states:

\begin{theorem}
Let $K$ be a field and $X$ an $n$-dimensional $K$-vector space. As a $\bk$-module, $\St(V)$ is generated by elements---called apartment classes---of the form $[v_1, \ldots, v_n]$, one for each ordered basis $v_1, \ldots, v_n$ of $V$, subject to the following relations: \begin{enumerate}
	\item $[v_1, \ldots, v_n] =\sgn(\sigma)[v_{\sigma(1)}, \ldots, v_{\sigma(n)}] $ for $\sigma$ a permutation.
	\item $[r v_1, v_2, \ldots, v_n] = [v_1, \ldots, v_n]$ for $r \in K^{\times}$.
	\item $\sum_i (-1)^i [v_0, v_1, v_2, \ldots, \hat v_i, \ldots, v_n] = 0$ where $v_0,\ldots, v_n$ are nonzero vectors and terms of the form $[v_0, v_1, v_2, \ldots, \hat v_i, \ldots, v_n] $ with $v_0, v_1, v_2, \ldots, \hat v_i, \ldots, v_n$ not a basis are omitted from the sum. 
\end{enumerate}

\noindent The $\GL(V)$-action is given by the formula $ g [v_1, \ldots, v_n]=[g v_1, \ldots, g v_n]$.
\end{theorem}

We need the product structure from \cite[Proposition 2.6]{MNP}:

\begin{proposition} Let $K$ be a field and let $V$ and $U$ be $K$-vector spaces of dimension $n$ and $m$ respectively.
The map $\St(V) \otimes_{\bk} \St(U)  \m \St(V \oplus U)$ given by \[ [v_1, \ldots, v_n] \otimes [u_1, \ldots, u_m]  \mapsto  [v_1, \ldots, v_n, u_1, \ldots, u_m] \] is well-defined and gives $\St$ the structure of a monoid object in $(\Mod_{\VB},\otimes)$.
\end{proposition} 

The chain complex $\overline{\mathcal B}^n_*(\St)$---the reduced bar complex of $\St$ in degree $n$---is given by that chain groups
\[ \overline{\mathcal B}^n_s(\St) = \bigoplus_{\substack{W_1 \oplus \dots \oplus W_s = K^n\\ W_i \neq 0}} \St(W_1) \otimes \dots \otimes \St(W_s)\]
and the differential $\partial\colon \overline{\mathcal B}^n_s( \St) \to \overline{\mathcal B}^n_{s-1}( \St)$ is given as the alternating sum $\partial = \sum_{1 \le i \le s-1} (-1)^i d_i$, where $d_i$ is the direct sum of the maps
\[ \St(W_1) \otimes \dots \otimes \St(W_i) \otimes \St(W_{i+1})\otimes \dots \otimes \St(W_s) \longrightarrow \St(W_1) \otimes \dots \otimes \St(W_i\oplus W_{i+1})\otimes \dots \otimes \St(W_s),\]
using the product structure.

In order to show \autoref{prop3.6}, we will define a filtration on $\overline{\mathcal B}^n_*(\St)$ and show that successive quotients have vanishing homology in degrees $*<n$. (Note that the chain complex is zero in degrees $*>n$.)

Given a subspace $W\subset K^n$ of dimension $d$, we say that an ordered basis $(w_1, \dots, w_d)$ of $W$ is a \emph{PBW-basis of $W$} if for $w_i$ the first nonzero entry, whose position we will denote by $s_i$, is $1\in K$, and $s_1< s_2 < \dots < s_d$. 

Despite there being many PBW-bases for $W$, the indices $s_1< s_2 < \dots < s_d$ are uniquely determined for each $W$. In particular, $s_{i}$ is the largest integer $k$ between $1$ and $n$ such that the intersection of $W$ and the subspace of $K^n$ whose first $k-1$ entries are zero is at least $(d-i+1)$-dimensional. Let us write
\[ S_W = (s_1,\dots, s_d).\]

We observe in \cite[Proposition 3.3]{MNP}, that the Solomon--Tits Theorem implies that the set of apartment classes 
\[ \{ [w_1,\dots,w_d] \in \St(W) \mid (w_1,\dots,w_d) \text{ is a PBW-basis of $W$}\}\]
coming from PBW-bases gives a basis of $\St(W)$ as a free $\bk$-module. This will also determine a chosen basis 
\[\{ [a_1] \otimes \dots \otimes [a_p] \mid a_i \text{ is a PBW-basis of $W_i$ for a decomposition $W_1 \oplus \dots \oplus W_p = K^n$}\}\]
of  $\overline{\mathcal B}^n_p(\St)$. To each (ordered) decomposition $D = (W_1, \dots, W_p)$ with $W_1 \oplus \dots \oplus W_p = K^n$, we associate the sequence
\[ S_D = S_{W_1}S_{W_2} \dots S_{W_p} \in \{1, \dots, n\}^n\]
concatenating the sequences $S_{W_1}, \dots, S_{W_p}$, which is a sequence of $n$ integers between $1$ and $n$.

Given a sequence $S = (s_1, \dots, s_n)\in \{1,\dots,n\}^n$, we define
\[I(S) =  (c_1,\dots, c_n, k) \in \{1, \dots,n\}^n \times \{0,\dots, {n \choose 2}\},\]
where $c_i$ is the number of $i$'s in $S$ and $k$ is the number of pairs $(i,j)$ with $1\le i < j \le n$ and $s_i > s_j$. We write $I_D$ for $I(S_D)$. We will use the lexicographic order $\le_{\mathrm{lex}}$ on $\{1, \dots,n\}^n \times \{0,\dots, {n \choose 2}\}$ to define a filtration 
\[ F_{I}\overline{\mathcal B}^n_p(\St) = \bigoplus_{\substack{D= (W_1, \dots, W_p) \\ W_1 \oplus \dots \oplus W_s = K^n\\ W_i \neq 0\\I_D \le_{\mathrm{lex}} I}} \St(W_1) \otimes \dots \otimes \St(W_s).\]

\begin{lemma}
$F_I\overline{\mathcal B}^n_*(\St)$ defines a subcomplex of $\overline{\mathcal B}^n_*(\St)$.
\end{lemma}

\begin{proof}
Let $D = (W_1, \dots, W_p)$ with $W_1 \oplus \dots \oplus W_p = K^n$ and $D_i = (W_1 , \dots ,(W_i\oplus W_{i+1}),  \dots ,W_p )$. We will show that $I_D \ge_{\mathrm{lex}} I_{D_i}$ to prove that

Note that the difference of the sequences $S_D$ and $S_{D_i}$ is the part $S_{W_i}S_{W_{i+1}}$ of $S_D$ and $S_{W_i \oplus W_{i+1}}$ of $S_{D_i}$. We will therefore analyze how this can change.

Case 1: Let us assume that $S_{W_i}$ and $S_{W_{i+1}}$ have no elements in common. Then $S_{W_i \oplus W_{i+1}}$ is precisely the union of the elements of $S_{W_i}$ and $S_{W_{i+1}}$ put into the correct order. In particular, the $c$'s in $I_D$ and $I_{D_i}$ are identical and $k$ cannot get larger. This proves the assertion in this case.

Case 2: Let us assume that $S_{W_i}$ and $S_{W_{i+1}}$ have an element in common, and let $k$ be the smallest element in common. Then $S_{W_i \oplus W_{i+1}}$ contains all elements of $S_{W_i}$ and $S_{W_{i+1}}$ that are at most $k$ (but in the correct order). But $k$ can only appear once and so a number bigger than $k$ has to replace one of them. This means that $c_j$ for $j<k$ stays the same and $c_k$ is decreased by at least one. Thus, in the lexicographical order $I_D > I_{D_i}$. This proves the assertion in this case.

Let us consider one simple example for Case 2, to make clear what happens. Say, we have $W_1 \oplus W_2 = K^3$ given by the PBW-bases
\[ W_1 = \mathrm{span}( w_1= \begin{bmatrix} 1\\2\\3\end{bmatrix}, w_2 = \begin{bmatrix} 0\\1\\4\end{bmatrix}), \quad W_2 = \mathrm{span}(  w_3 = \begin{bmatrix} 0\\1\\6\end{bmatrix}).\]
Then $w_2-w_3 = \begin{bmatrix} 0\\0\\-2\end{bmatrix}$ has the first nonzero entry in the third place and can be made into the PBW-basis 
\[w_1= \begin{bmatrix} 1\\2\\3\end{bmatrix}, \quad w_2 = \begin{bmatrix} 0\\1\\4\end{bmatrix}, \quad -\frac12(w_2-w_3) = \begin{bmatrix} 0\\0\\1\end{bmatrix}\]
of $K^3$. In particular, $S_{W_1}S_{W_2} = (1,2,2)$ turned into $S_{K^3}= (1,2,3)$.
\end{proof}

It remains to prove that 
\[H_s(F_I\overline{\mathcal B}^n_*(\St)/ F_{I'}\overline{\mathcal B}^n_*(\St)) = 0\quad \text{for $s<n$,}\]
where $I'$ is the element in $\{1, \dots,n\}^n \times \{0,\dots, {n \choose 2}\}$ one smaller than $I$ in lexicographical order. To this end, we will define 
\[ \Phi\colon F_I\overline{\mathcal B}^n_p(\St) \longrightarrow F_I\overline{\mathcal B}^n_{p+1}(\St)\]
such that 
\[ \Phi \partial + \partial \Phi = \id\]
on 
\[ F_I\overline{\mathcal B}^n_p(\St)/ F_{I'}\overline{\mathcal B}^n_p(\St)\]
if $p<n$. Such a $\Phi$ is a homotopy between the identity map and a map $f$ from $F_I\overline{\mathcal B}^n_p(\St)/ F_{I'}\overline{\mathcal B}^n_p(\St)$ to itself, where $f$ is zero on all degrees $p<n$. So the existence of such a $\Phi$ implies the assertion of \autoref{prop3.6}.

We now define $\Phi$. For $S\in \{1,\dots,n\}^n$ for which there exists an $i$ such that $s_i<s_{i+1}$, let
\[ k_S = \min\{s \mid s = s_i < s_{i+1}\}\]
and
\[ j_S = \min\{i \mid k_S = s_i < s_{i+1}\}.\]
Let
\[ a_1 \otimes \dots \otimes a_p \in \St(W_1) \otimes \dots \otimes \St(W_p) \subset \overline{\mathcal B}^n_p(\St)\]
for a decomposition $D = (W_1, \dots, W_p)$ with $W_1 \oplus \dots\oplus W_p = K^n$ and $p<n$ be one of the chosen basis elements. In particular,
\[ a_i = [w_{\dim W_1 + \dots + \dim W_{i-1} +1}, \dots, w_{\dim W_1 + \dots + \dim W_{i-1} +\dim W_i}]\]
is the apartment coming from a PBW-basis of $W_i$. Then $k_{S_D}$ and $j_{S_D}$ exist because $p<n$, and say $i_0$ is the number such that
\[ \dim W_1 + \dots + \dim W_{i_0-1} +1 \le j_{S_D} \le \dim W_1 + \dots + \dim W_{i-1} +\dim W_{i_0}.\]
Note that then
\[  j_{S_D} = \dim W_1 + \dots + \dim W_{i_0-1} +1\]
because $s_{j_{S_D}}$ has to be minimal. If $\dim W_{i_0} =1$, we define
\[ \Phi(a_1 \otimes \dots \otimes a_p) =0.\]
If $\dim W_{i_0} >1$, we split off the first vector of $a_{i_0}$: let 
\[a'_{i_0} = [w_{\dim W_1 + \dots + \dim W_{i_0-1} +1}] \in \St( \mathrm{span}( w_{\dim W_1 + \dots + \dim W_{i_0-1} +1}))\]
and
\begin{multline*}a''_{i_0} = [w_{\dim W_1 + \dots + \dim W_{i_0-1} +2},\dots, w_{\dim W_1 + \dots + \dim W_{i-1} +\dim W_{i_0}}] \\\in \St( \mathrm{span}( w_{\dim W_1 + \dots + \dim W_{i_0-1} +2},\dots, w_{\dim W_1 + \dots + \dim W_{i-1} +\dim W_{i_0}})).\end{multline*}
We then define
\[ \Phi(a_1 \otimes \dots \otimes a_p) = (-1)^{i_0} \cdot a_1 \otimes \dots \otimes a_{i_0-1} \otimes a'_{i_0} \otimes a''_{i_0} \otimes a_{i_0+1} \otimes \dots \otimes a_p \in \overline{\mathcal B}^n_{p+1}(\St).\]
Observe that $S$ of the corresponding decompositions does not change, so this restricts to a map
\[ \Phi\colon F_I\overline{\mathcal B}^n_p(\St) \longrightarrow F_I\overline{\mathcal B}^n_{p+1}(\St).\]

It remains to prove that 
\[ \Phi \partial + \partial \Phi = \id\]
on 
\[ F_I\overline{\mathcal B}^n_p(\St)/ F_{I'}\overline{\mathcal B}^n_p(\St)\]
if $p<n$.

We will calculate $\Phi \partial + \partial \Phi$ in two cases as we also defined $\Phi$ in two cases.

First, assume that 
\[ a = a_1 \otimes \dots \otimes a_p \in \St(W_1) \otimes \dots \otimes \St(W_p) \subset F_I\overline{\mathcal B}^n_p(\St)\]
is as above with the assumption that $\dim W_{i_0} = 1$. Then $\Phi(a) = 0$ and so of course $\partial \Phi(a) =0$. On the other hand
\[ \partial(a) = \sum_{1 \le i \le p-1} (-1)^i a_1 \otimes \dots \otimes a_ia_{i+1} \otimes \dots \otimes a_p.\]
Note that in $F_I\overline{\mathcal B}^n_p(\St)/ F_{I'}\overline{\mathcal B}^n_p(\St)$, we only need to consider those summands whose $S_{D_i}$ does not change and thus $a_ia_{i+1}$ is still an apartment coming from a PBW-basis---otherwise $I_{D_i} <_{\mathrm{lex}} I_D$.
Now  
\[ \Phi(a_1 \otimes \dots \otimes a_ia_{i+1} \otimes \dots \otimes a_p) = 0\]
unless $i=i_0$, because otherwise the vector $w_{j_{S_D}}$ is still in the Steinberg module of a 1-dimensional subspace. And
\[ \Phi(a_1 \otimes \dots \otimes a_{i_0}a_{i_0+1} \otimes \dots \otimes a_p) = (-1)^{i_0}a_1 \otimes \dots \otimes a_p. \]
Using all of this information, we get that 
\[ (\partial \Phi + \Phi \partial)(a) = 0 + (-1)^{i_0}(-1)^{i_0} a = a\]
in the first case.

Second, we assume that $\dim W_{i_0} >1$. Then
\[ \Phi(a) = (-1)^{i_0} \cdot a_1 \otimes \dots \otimes a_{i_0-1} \otimes a'_{i_0} \otimes a''_{i_0} \otimes a_{i_0+1} \otimes \dots \otimes a_p.\]
It follows that
\begin{align*} &\partial (a_1 \otimes \dots \otimes a_{i_0-1} \otimes a'_{i_0} \otimes a''_{i_0} \otimes a_{i_0+1} \otimes \dots \otimes a_p)\\
 =& \phantom{+} \sum_{1 \le i \le i_0-2} (-1)^i a_1 \otimes \dots \otimes a_ia_{i+1} \otimes \dots \otimes a'_{i_0} \otimes a''_{i_0} \otimes\dots \otimes a_p\\
 & + (-1)^{i_0-1} a_1 \otimes \dots\otimes a_{i_0-1}a'_{i_0} \otimes a''_{i_0} \otimes\dots \otimes a_p\\
& + (-1)^{i_0} a_1 \otimes \dots \otimes a_p\\
 & + (-1)^{i_0+1} a_1 \otimes \dots\otimes a'_{i_0} \otimes a''_{i_0}a_{i_0+1} \otimes\dots \otimes a_p\\
 &+ \sum_{i_0+1 \le i \le p} (-1)^{i+1} a_1  \otimes \dots \otimes a'_{i_0} \otimes a''_{i_0} \otimes \dots \otimes a_ia_{i+1}\otimes\dots \otimes a_p.
 \end{align*}
 Note that the summand 
 \[ a_1 \otimes \dots\otimes a_{i_0-1}a'_{i_0} \otimes a''_{i_0} \otimes\dots \otimes a_p\]
 is zero in $F_I\overline{\mathcal B}^n_p(\St)/ F_{I'}\overline{\mathcal B}^n_p(\St)$ because $s$ of the last vector in $a_{i_0-1}$ has to be at least as big as the $s$ of the vector in $a'_{i_0}$ by the minimality condition.

On the other hand, 
\[ \partial(a) = \sum_{1 \le i \le p-1} (-1)^i a_1 \otimes \dots \otimes a_ia_{i+1} \otimes \dots \otimes a_p.\]
Note that here the summand
\[ a_1 \otimes \dots \otimes a_{i_0-1}a_{i_0} \otimes \dots \otimes a_p\]
is zero in $F_I\overline{\mathcal B}^n_p(\St)/ F_{I'}\overline{\mathcal B}^n_p(\St)$ as well.
Before applying $\Phi$, we will again observe that all summands 
\[ a_1 \otimes \dots \otimes a_ia_{i+1} \otimes \dots \otimes a_p\]
that are nonzero in $F_I\overline{\mathcal B}^n_p(\St)/ F_{I'}\overline{\mathcal B}^n_p(\St)$ have the same $S$ as $a$.
For those,
\[ \Phi(a_1 \otimes \dots \otimes a_ia_{i+1} \otimes \dots \otimes a_p) = (-1)^{i_0-1}\cdot a_1 \otimes \dots \otimes a_ia_{i+1} \otimes\dots \otimes a'_{i_0} \otimes a''_{i_0} \otimes \dots \otimes a_p\]
for $i< i_0 -1$,
\[ \Phi(a_1 \otimes \dots \otimes a_ia_{i+1} \otimes \dots \otimes a_p) = (-1)^{i_0}\cdot a_1 \otimes \dots \otimes a'_{i_0} \otimes a''_{i_0}a_{i_0+1} \otimes \dots \otimes a_p\]
for $i=i_0$, and
\[ \Phi(a_1 \otimes \dots \otimes a_ia_{i+1} \otimes \dots \otimes a_p) =(-1)^{i_0}\cdot a_1 \otimes \dots \otimes a'_{i_0} \otimes a''_{i_0} \otimes \dots\otimes a_ia_{i+1} \otimes\dots  \otimes a_p\]
for $i> i_0$.

We observe that the summands of $\partial\Phi(a)$ and $\Phi\partial(a)$ are cancelling exactly besides the summand
\[  (-1)^{i_0}\cdot (-1)^{i_0}\cdot a_1 \otimes \dots \otimes a_p.\]

This finishes the proof of \autoref{prop3.6}.

\begin{remark}
Another way to prove that 
\[H_s(F_I\overline{\mathcal B}^n_*(\St)/ F_{I'}\overline{\mathcal B}^n_*(\St)) = 0\quad \text{for $s<n$}\]
is by identifying the relative chains with a direct sum of $(n-p+1)$-shifted reduced simplicial chains of $p$-simplices with $p\ge -1$.  But we don't spell that out here, to stay close to the original structure of the proof.
\end{remark}

\bibliographystyle{amsalpha}
\bibliography{CodimOne}

\end{document}